\documentclass[11pt, a4paper]{article}
\usepackage{mathrsfs, mathtools, amsthm, amssymb, thm-restate}  
\usepackage{paralist, tablists}  

\usepackage{graphicx, xcolor, tikz}  
\usetikzlibrary{calc, shapes, decorations.pathreplacing, decorations.markings}  
\usepackage{caption}  
\usepackage[labelformat=simple]{subcaption}  

\usepackage[left=25mm, top=25mm, bottom=25mm, right=25mm]{geometry}  
\usepackage[T1]{fontenc} 

\usepackage[inline]{enumitem}  
\usepackage[square, numbers, sort&compress]{natbib}  
\usepackage[
  bookmarks=true,                
  bookmarksnumbered=true,        
  bookmarksopen=true,            
  plainpages=false,              
  colorlinks=true,               
  linkcolor=blue,                
  citecolor=black,               
  anchorcolor=green,             
  urlcolor=blue,                 
  hyperindex=true                
]{hyperref} 

\newtheorem{theorem}{Theorem}[section] 
\newtheorem{lemma}[theorem]{Lemma}
\newtheorem{corollary}[theorem]{Corollary}

\theoremstyle{definition}
\newtheorem{definition}{Definition}

\newtheorem{proposition}{Proposition}

\makeatletter
\def\th@plain{%
  \upshape 
}
\makeatother

\newcommand{\etal}{et~al.\ }
\newcommand{\ie}{i.e.,\ }

\newcommand{\ee}{\mathrm{EE}}
\renewcommand{\oe}{\mathrm{OE}}
\usepackage{marvosym,wasysym}
\usepackage{array}
\usepackage{bm}  

\makeatletter
\renewenvironment{proof}[1][\proofname]{\par
  \pushQED{\qed}%
  \normalfont \topsep6\p@\@plus6\p@\relax
  \trivlist
  \item[\hskip\labelsep
        \bfseries
    #1\@addpunct{.}]\ignorespaces
}{%
  \popQED\endtrivlist\@endpefalse
}
\makeatother

\usepackage[capitalise]{cleveref}  
\crefname{claim}{Claim}{Claims}

\tikzset{
  on each segment/.style={
    decorate,
    decoration={
      show path construction,
      moveto code={},
      lineto code={
        \path [#1] (\tikzinputsegmentfirst) -- (\tikzinputsegmentlast);
      },
      curveto code={
        \path [#1] (\tikzinputsegmentfirst)
        .. controls (\tikzinputsegmentsupporta) and (\tikzinputsegmentsupportb) ..
        (\tikzinputsegmentlast);
      },
      closepath code={
        \path [#1] (\tikzinputsegmentfirst) -- (\tikzinputsegmentlast);
      },
    },
  },
  mid arrow/.style={postaction={decorate,decoration={
        markings,
        mark=at position .55 with {\arrow[#1]{stealth}}
      }}},
}

\begin{document}

\title{Toroidal graphs without $K_{5}^{-}$ and 6-cycles}
\author{Ping Chen\footnote{School of Mathematics and Statistics, Henan University, Kaifeng, 475004, P. R. China} \and Tao Wang\footnote{Center for Applied Mathematics, Henan University, Kaifeng, 475004, P. R. China. \tt Corresponding author: wangtao@henu.edu.cn; https://orcid.org/0000-0001-9732-1617}}
\date{}
\maketitle
\begin{abstract}
Cai \etal proved that a toroidal graph $G$ without $6$-cycles is $5$-choosable, and proposed the conjecture that $\textsf{ch}(G) = 5$ if and only if $G$ contains a $K_{5}$ [J. Graph Theory 65 (2010) 1--15], where $\textsf{ch}(G)$ is the choice number of $G$. However, Choi later disproved this conjecture, and proved that toroidal graphs without $K_{5}^{-}$ (a $K_{5}$ missing one edge) and $6$-cycles are $4$-choosable [J. Graph Theory 85 (2017) 172--186]. In this paper, we provide a structural description, for toroidal graphs without $K_{5}^{-}$ and $6$-cycles. Using this structural description, we strengthen Choi's result in two ways: (I) we prove that such graphs have weak degeneracy at most three (nearly $3$-degenerate), and hence their DP-paint numbers and DP-chromatic numbers are at most four; (II) we prove that such graphs have Alon-Tarsi numbers at most $4$. Furthermore, all of our results are sharp in some sense. 

Keywords: Toroidal graphs; Weak degeneracy; Alon-Tarsi number; DP-coloring 

MSC2020: 05C15
\end{abstract}

\section{Introduction}
A \emph{$k$-list assignment} of a graph $G$ is a mapping $L$ that assigns a list $L(v)$ of $k$ admissible colors to each vertex $v$ in $G$. An \emph{$L$-coloring} of $G$ is a proper coloring $\phi$ of $G$ such that $\phi(v) \in L(v)$ for all $v\in V(G)$. A graph $G$ is \emph{$k$-choosable} if $G$ admits an $L$-coloring for each $k$-list assignment $L$. The \emph{choice number}, or \emph{list chromatic number}, $\textsf{ch}(G)$ is the smallest integer $k$ such that $G$ is $k$-choosable.

Thomassen \cite{MR1290638} proved that every planar graph is $5$-choosable, and Voigt \cite{MR1235909} constructed a planar graph that is not $4$-choosable. B\"{o}hme \etal \cite{MR1722227} proved that every toroidal graph is $7$-choosable, and a toroidal graph $G$ has $\textsf{ch}(G) = 7$ if and only if $K_{7} \subseteq G$. 

Cai \etal \cite{MR2682511} investigated the choosability of toroidal graphs without short cycles. They also conjectured that every toroidal graph without $K_{5}$ and $6$-cycles is $4$-choosable, but Choi \cite{MR3634481} disproved this conjecture, and proved a weak version of it. A $K_{5}^{-}$ is a $K_{5}$ missing one edge.

\begin{theorem}[Choi \cite{MR3634481}]
Every toroidal graph without $K_{5}^{-}$ and $6$-cycles is $4$-choosable. 
\end{theorem}

Dvo\v{r}\'{a}k and Postle \cite{MR3758240} introduced the concept of DP-coloring, which is a generalization of list coloring. They observed that every $k$-DP-colorable graph is also a $k$-choosable graph. However, it is not known whether toroidal graphs without $K_{5}^{-}$ and $6$-cycles are 4-DP-colorable. In this paper, we focus on this class of graphs, and provide the following structural result that can be used to answer this question positively. A \emph{configuration} is a subgraph possibly with some degree restrictions in the host graph.

\begin{theorem}\label{SR}
Let $G$ be a connected toroidal graph without $K_{5}^{-}$ and $6$-cycles. Then one of the following holds:
\begin{enumerate}
\item The minimum degree is at most three.
\item There is an induced subgraph isomorphic to the configuration in \cref{K5--}.
\item There is an induced subgraph isomorphic to the configuration in \cref{Kite}.
\item There is an induced subgraph isomorphic to the configuration in \cref{House}. 
\end{enumerate}
\end{theorem}

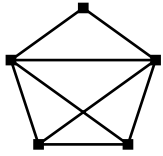
\begin{figure}%
\centering
\begin{tikzpicture}[line width = 1pt]
\def\s{1}
\foreach \ang in {1, 2, 3, 4, 5}
{
\def\pointname{v\ang}
\coordinate (\pointname) at ($(\ang*360/5-54:\s)$);}
\draw (v1)--(v2)--(v3)--(v4)--(v5)--cycle;
\draw (v5)--(v3)--(v1)--(v4);
\foreach \ang in {1, 2, 3, 4, 5}
{
\node[rectangle, inner sep = 1.5, fill = black, draw] () at (v\ang) {};
}
\end{tikzpicture}
\caption{A configuration. Here and in all figures below, a solid quadrilateral represents a $4$-vertex.}
\label{K5--}
\end{figure}

\begin{figure}%
\centering
\begin{tikzpicture}[line width = 1pt]%
\def\s{1}
\coordinate (E) at (\s, 0);
\coordinate (N) at (0, \s);
\coordinate (W) at (-\s, 0);
\coordinate (S) at (0, -\s);
\draw (N)--(W)--(S)--(E)--cycle;
\draw (N)--(S);
\node[rectangle, inner sep = 1.5, fill, draw] () at (E) {};
\node[rectangle, inner sep = 1.5, fill, draw] () at (N) {};
\node[rectangle, inner sep = 1.5, fill, draw] () at (W) {};
\node[rectangle, inner sep = 1.5, fill, draw] () at (S) {};
\end{tikzpicture}
\caption{Kite graph.}
\label{Kite}
\end{figure}
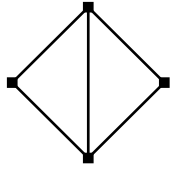

\begin{figure}%
\centering
\begin{tikzpicture}[line width = 1pt]
\def\s{1}
\coordinate (A) at (45:\s);
\coordinate (B) at (135:\s);
\coordinate (C) at (225:\s);
\coordinate (D) at (-45:\s);
\coordinate (H) at (90:1.414*\s);
\draw (A)--(H)--(B)--(C)--(D)--cycle;
\draw (A)--(B);
\foreach \i in {A,B,C,D,H}{
            \node[rectangle, inner sep = 1.5, fill, draw] () at (\i) {};
}
\end{tikzpicture}
\caption{House graph.}
\label{House}
\end{figure}
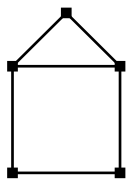

Using this structural description, we consider two graph parameters, \emph{weak degeneracy} and \emph{Alon-Tarsi number}.

In a greedy algorithm, when a vertex $u$ is colored with $\alpha$, we should remove the colors conflict with $\alpha$ from the lists of colors available to the neighbors of $u$. Then the list size of each neighbor of $u$ may decrease by 1. If throughout the coloring process, each vertex has at least one available color when it needs to be colored, then we can obtain a proper (DP-) coloring. If at some steps, one can ``save'' some colors for a neighbor of $u$, we may obtain a better upper bound for the (DP-) chromatic number. This idea was used by Bernshteyn and Lee \cite{MR4606413} to define a notion of \emph{weak degeneracy}.

\begin{definition}[\textsf{Delete} operation]
Let $G$ be a graph and $f : V(G) \longrightarrow \mathbb{N}$ be a function. For a vertex $u \in V(G)$, the operation \textsf{Delete}$(G, f, u)$ outputs the graph $G' = G - u$ and the function $f': V(G') \longrightarrow \mathbb{Z}$ given by 
\begin{align*}
f'(v) \coloneqq
&\begin{cases}
f(v) - 1, & \text{if $uv \in E(G)$};\\[0.3cm]
f(v), & \text{otherwise}.
\end{cases}
\end{align*}
An application of the operation \textsf{Delete} is \emph{legal} if the resulting function $f'$ is nonnegative. 
\end{definition}

\begin{definition}[\textsf{DeleteSave} operation]
Let $G$ be a graph and $f : V(G) \longrightarrow \mathbb{N}$ be a function. For a pair of adjacent vertices $u, w \in V(G)$, the operation \textsf{DeleteSave}$(G, f, u, w)$ outputs the graph $G' = G - u$ and the function $f' : V(G') \longrightarrow \mathbb{Z}$ given by
\begin{align*}
f'(v) \coloneqq
&\begin{cases}
f(v) - 1, & \text{if $uv \in E(G)$ and $v \neq w$};\\[0.3cm]
f(v), & \text{otherwise}.
\end{cases}
\end{align*}
An application of the operation \textsf{DeleteSave} is \emph{legal} if $f(u) > f(w)$ and the resulting function $f'$ is nonnegative. 
\end{definition}

A graph $G$ is \emph{weakly $f$-degenerate} if it is possible to remove all vertices from $G$ by a sequence of legal applications of the operations \textsf{Delete} or \textsf{DeleteSave}. A graph is \emph{$f$-degenerate} if it is weakly $f$-degenerate with no the \textsf{DeleteSave} operation. Given a nonnegative integer $d$, we say that $G$ is \emph{weakly $d$-degenerate} if it is weakly $f$-degenerate with respect to the constant function of value $d$. We say that $G$ is \emph{$d$-degenerate} if it is $f$-degenerate with respect to the constant function of value $d$. The \emph{weak degeneracy} of $G$, denote by $\textsf{wd}(G)$, is the minimum integer $d$ such that $G$ is weakly $d$-degenerate. The \emph{degeneracy} of $G$, denote by $\textsf{d}(G)$, is the minimum integer $d$ such that $G$ is $d$-degenerate. 

Bernshteyn and Lee \cite{MR4606413} provided the following inequalities on several graph parameters. 

\begin{proposition}\label{prop}
For any graph $G$, we always have 
\begin{equation*}
\chi(G) \leq \textsf{ch}(G) \leq \chi_{\textsf{DP}}(G) \leq \chi_{\textsf{DPP}}(G) \leq \textsf{wd}(G) + 1 \leq \textsf{d}(G) + 1, 
\end{equation*}
where $\chi_{\textsf{DP}}(G)$ is the DP-chromatic number of $G$, and $\chi_{\textsf{DPP}}(G)$ is the DP-paint number of $G$. 
\end{proposition}

Wang \etal \cite{Wang2019+} studied three families of graphs with weak degeneracy at most three.

\begin{figure}%
\centering
\subcaptionbox{\label{fig:subfig:a}}
{\begin{tikzpicture}[line width = 1pt, scale = 0.8]
\def\s{1.3}
\coordinate (A) at (\s, 0);
\coordinate (B) at (60:\s);
\coordinate (O) at (0, 0);
\coordinate (C) at (-150:\s);
\coordinate (D) at (0,-\s);
\coordinate (E) at (\s, -\s);
\draw (A)--(B)--(O)--(C)--(D)--(E)--cycle;
\draw (A)--(O)--(D);
\node[circle, inner sep = 1.5, fill = white, draw] () at (O) {};
\node[circle, inner sep = 1.5, fill = white, draw] () at (A) {};
\node[circle, inner sep = 1.5, fill = white, draw] () at (B) {};
\node[circle, inner sep = 1.5, fill = white, draw] () at (C) {};
\node[circle, inner sep = 1.5, fill = white, draw] () at (D) {};
\node[circle, inner sep = 1.5, fill = white, draw] () at (E) {};
\end{tikzpicture}}\hspace{1.5cm}
\subcaptionbox{\label{fig:subfig:b}}
{\begin{tikzpicture}[line width = 1pt, scale = 0.8]
\def\s{1}
\coordinate (A) at (\s, 0);
\coordinate (B) at (0.5*\s, 0.7*\s);
\coordinate (O) at (0, 0);
\coordinate (E) at (\s, -\s);
\coordinate (D) at (0,-\s);
\coordinate (C) at (0.5*\s, -1.7*\s);
\draw (A)--(B)--(O)--(D)--(C)--(E)--cycle;
\draw (A)--(O);
\draw (D)--(E);
\node[circle, inner sep = 1.5, fill = white, draw] () at (O) {};
\node[circle, inner sep = 1.5, fill = white, draw] () at (A) {};
\node[circle, inner sep = 1.5, fill = white, draw] () at (B) {};
\node[circle, inner sep = 1.5, fill = white, draw] () at (C) {};
\node[circle, inner sep = 1.5, fill = white, draw] () at (D) {};
\node[circle, inner sep = 1.5, fill = white, draw] () at (E) {};
\end{tikzpicture}}\hspace{1.5cm}
\subcaptionbox{\label{fig:subfig:433}}
{\begin{tikzpicture}[line width = 1pt, scale = 0.8]
\def\s{1.2}
\coordinate (O) at (0, 0);
\coordinate (A) at (0, -\s);
\coordinate (B) at (-\s,-\s);
\coordinate (C) at (-\s,0);
\coordinate (D) at (135:\s);
\coordinate (E) at (0,\s);
\draw (O)--(A)--(B)--(C)--(D)--(E)--cycle;
\draw (C)--(O)--(D);
\node[circle, inner sep = 1.5, fill = white, draw] () at (O) {};
\node[circle, inner sep = 1.5, fill = white, draw] () at (A) {};
\node[circle, inner sep = 1.5, fill = white, draw] () at (B) {};
\node[circle, inner sep = 1.5, fill = white, draw] () at (C) {};
\node[circle, inner sep = 1.5, fill = white, draw] () at (D) {};
\node[circle, inner sep = 1.5, fill = white, draw] () at (E) {};
\end{tikzpicture}}\hspace{1.5cm}
\subcaptionbox{\label{fig:subfig:3444}}
{\begin{tikzpicture}[line width = 1pt, scale = 0.8]
\def\s{1}
\coordinate (O) at (0, 0);
\coordinate (E) at (\s, 0);
\coordinate (W) at (-\s,0);
\coordinate (N) at (0,\s);
\coordinate (S) at (0,-\s);
\coordinate (NW) at (-\s,\s);
\coordinate (SW) at (-\s,-\s);
\coordinate (SE) at (\s,-\s);
\draw (E)--(N)--(NW)--(SW)--(SE)--cycle;
\draw (N)--(S); 
\draw (E)--(W); 
\node[circle, inner sep = 1.5, fill = white, draw] () at (O) {};
\node[circle, inner sep = 1.5, fill = white, draw] () at (E) {};
\node[circle, inner sep = 1.5, fill = white, draw] () at (W) {};
\node[circle, inner sep = 1.5, fill = white, draw] () at (N) {};
\node[circle, inner sep = 1.5, fill = white, draw] () at (S) {};
\node[circle, inner sep = 1.5, fill = white, draw] () at (SE) {};
\node[circle, inner sep = 1.5, fill = white, draw] () at (SW) {};
\node[circle, inner sep = 1.5, fill = white, draw] () at (NW) {};
\end{tikzpicture}}\\
\subcaptionbox{\label{fig:subfig:533}}
{\begin{tikzpicture}[line width = 1pt, scale = 0.8]
\def\s{1}
\foreach \ang in {1, 2, 3, 4, 5}
{
\def\pointname{v\ang}
\coordinate (\pointname) at ($(\ang*360/5-54:\s)$);}
\coordinate (S) at ($(v2)!1!60:(v1)$);
\coordinate (S') at ($(v2)!1!60:(S)$);
\draw (v1)--(v2)--(v3)--(v4)--(v5)--cycle;
\draw (v1)--(S)--(v2);
\draw (S)--(S')--(v2);
\node[circle, inner sep = 1.5, fill = white, draw] () at (S) {};
\node[circle, inner sep = 1.5, fill = white, draw] () at (S') {};
\foreach \ang in {1, 2, 3, 4, 5}
{
\node[circle, inner sep = 1.5, fill = white, draw] () at (v\ang) {};
}
\end{tikzpicture}}\hspace{1cm}
\subcaptionbox{\label{fig:subfig:534f}}
{\begin{tikzpicture}[line width = 1pt, scale = 0.8]
\def\s{1}
\foreach \ang in {1, 2, 3, 4, 5}
{
\def\pointname{v\ang}
\coordinate (\pointname) at ($(\ang*360/5-54:\s)$);
}
\coordinate (S) at ($(v2)!1!60:(v1)$);
\coordinate (S') at ($(v2)!1!90:(S)$);
\coordinate (S'') at ($(S)!1!-90:(v2)$);
\draw (v1)--(v2)--(v3)--(v4)--(v5)--cycle;
\draw (v1)--(S)--(v2);
\draw (S)--(S'')--(S')--(v2);
\node[circle, inner sep = 1.5, fill = white, draw] () at (S) {};
\node[circle, inner sep = 1.5, fill = white, draw] () at (S') {};
\node[circle, inner sep = 1.5, fill = white, draw] () at (S'') {};
\foreach \ang in {1, 2, 3, 4, 5}
{
\node[circle, inner sep = 1.5, fill = white, draw] () at (v\ang) {};
}
\end{tikzpicture}}\hspace{1cm}
\subcaptionbox{\label{fig:subfig:g}}
{\begin{tikzpicture}[line width = 1pt, scale = 0.8]
\def\s{1.2}
\coordinate (O) at (0, 0);
\coordinate (v1) at (0:\s);
\coordinate (v2) at (60:\s);
\coordinate (v3) at (120:\s);
\coordinate (v4) at (180:\s);
\coordinate (OO) at ($(v3)!1!60:(v2)$);
\coordinate (A) at ($(OO)!1!-60:(v3)$);
\coordinate (B) at ($(OO)!1!-60:(A)$);
\coordinate (C) at ($(OO)!1!-60:(B)$);
\coordinate (D) at ($(OO)!1!-60:(C)$);
\draw (v1)--(v2)--(v3)--(v4)--(O)--cycle;
\draw (O)--(v2);
\draw (O)--(v3);
\draw (v3)--(A)--(B)--(C)--(D)--(v2);
\node[circle, inner sep = 1.5, fill = white, draw] () at (O) {};
\node[circle, inner sep = 1.5, fill = white, draw] () at (v3) {};
\node[circle, inner sep = 1.5, fill = white, draw] () at (v1) {};
\node[circle, inner sep = 1.5, fill = white, draw] () at (v2) {};
\node[circle, inner sep = 1.5, fill = white, draw] () at (v4) {};
\node[circle, inner sep = 1.5, fill = white, draw] () at (A) {};
\node[circle, inner sep = 1.5, fill = white, draw] () at (B) {};
\node[circle, inner sep = 1.5, fill = white, draw] () at (C) {};
\node[circle, inner sep = 1.5, fill = white, draw] () at (D) {};
\end{tikzpicture}}\hspace{1cm}
\subcaptionbox{\label{fig:subfig:h}}
{\begin{tikzpicture}[line width = 1pt, scale = 0.8]
\def\s{1.2}
\coordinate (O) at (0, 0);
\coordinate (v1) at (0:\s);
\coordinate (v2) at (60:\s);
\coordinate (v3) at (120:\s);
\coordinate (v4) at (180:\s);
\coordinate (OO) at ($(v2)!1!60:(v1)$);
\coordinate (A) at ($(OO)!1!-60:(v2)$);
\coordinate (B) at ($(OO)!1!-60:(A)$);
\coordinate (C) at ($(OO)!1!-60:(B)$);
\coordinate (D) at ($(OO)!1!-60:(C)$);
\draw (v1)--(v2)--(v3)--(v4)--(O)--cycle;
\draw (O)--(v2);
\draw (O)--(v3);
\draw (v2)--(A)--(B)--(C)--(D)--(v1);
\node[circle, inner sep = 1.5, fill = white, draw] () at (O) {};
\node[circle, inner sep = 1.5, fill = white, draw] () at (v3) {};
\node[circle, inner sep = 1.5, fill = white, draw] () at (v1) {};
\node[circle, inner sep = 1.5, fill = white, draw] () at (v2) {};
\node[circle, inner sep = 1.5, fill = white, draw] () at (v4) {};
\node[circle, inner sep = 1.5, fill = white, draw] () at (A) {};
\node[circle, inner sep = 1.5, fill = white, draw] () at (B) {};
\node[circle, inner sep = 1.5, fill = white, draw] () at (C) {};
\node[circle, inner sep = 1.5, fill = white, draw] () at (D) {};
\end{tikzpicture}}
\caption{Forbidden configurations in planar graphs.}
\label{FIGPAIRWISE3456}
\end{figure}
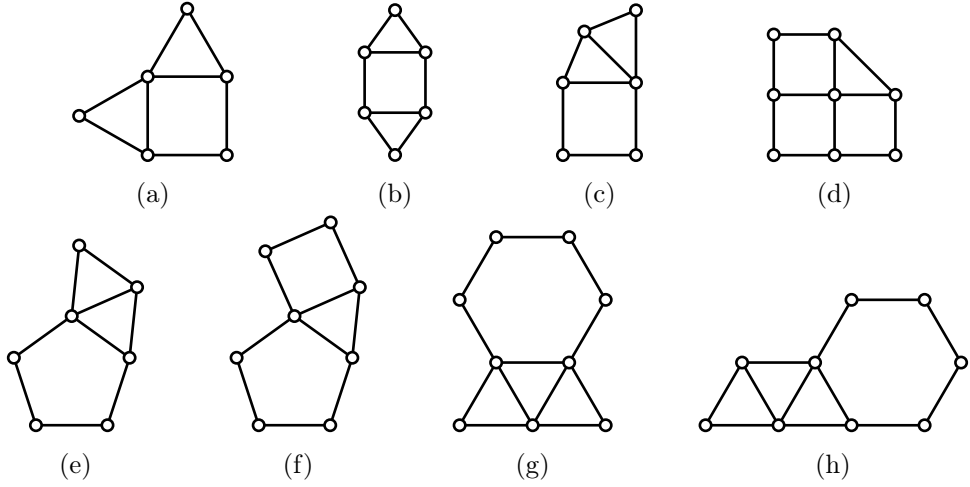

\begin{figure}%
\centering
\subcaptionbox{\label{fig:subfig:TA}}{\begin{tikzpicture}[line width = 1pt, scale = 0.8]
\coordinate (A) at (45:1);
\coordinate (B) at (135:1);
\coordinate (C) at (225:1);
\coordinate (D) at (-45:1);
\coordinate (H) at (90:1.414);
\draw (A)--(H)--(B)--(C)--(D)--cycle;
\draw (A)--(B);
\node[circle, inner sep = 1.5, fill = white, draw] () at (A) {};
\node[circle, inner sep = 1.5, fill = white, draw] () at (B) {};
\node[circle, inner sep = 1.5, fill = white, draw] () at (C) {};
\node[circle, inner sep = 1.5, fill = white, draw] () at (D) {};
\node[circle, inner sep = 1.5, fill = white, draw] () at (H) {};
\end{tikzpicture}}\hspace{1.5cm}
\subcaptionbox{\label{fig:subfig:TB}}{\begin{tikzpicture}[line width = 1pt, scale = 0.8]
\coordinate (A) at (45:1);
\coordinate (B) at (135:1);
\coordinate (C) at (225:1);
\coordinate (D) at (-45:1);
\coordinate (H) at (90:1.414);
\coordinate (X) at ($(A)+(0, 1)$);
\coordinate (Y) at (45:2);
\draw (A)--(Y)--(X)--(H)--(B)--(C)--(D)--cycle;
\draw (X)--(A)--(H);
\node[circle, inner sep = 1.5, fill = white, draw] () at (A) {};
\node[circle, inner sep = 1.5, fill = white, draw] () at (B) {};
\node[circle, inner sep = 1.5, fill = white, draw] () at (C) {};
\node[circle, inner sep = 1.5, fill = white, draw] () at (D) {};
\node[circle, inner sep = 1.5, fill = white, draw] () at (H) {};
\node[circle, inner sep = 1.5, fill = white, draw] () at (X) {};
\node[circle, inner sep = 1.5, fill = white, draw] () at (Y) {};
\end{tikzpicture}}
\caption{Forbidden configurations in toroidal graphs.}
\label{A345}
\end{figure}
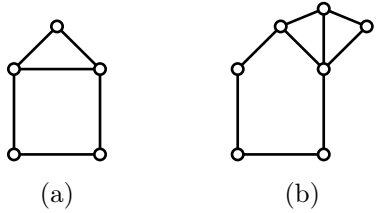

\begin{theorem}[Wang \etal \cite{Wang2019+}]\text{}
\begin{enumerate}
\item Every planar graph without any configuration in \cref{FIGPAIRWISE3456} is weakly $3$-degenerate. 

\item Every toroidal graph without any configuration in \cref{A345} is weakly $3$-degenerate. 

\item Every planar graph without intersecting $5$-cycles is weakly $3$-degenerate. 
\end{enumerate}
\end{theorem}

We say that two cycles are \emph{adjacent} if they share at least one edge, and say that they are \emph{normally adjacent} if their intersection is isomorphic to $K_{2}$. Recently, Han \etal \cite{MR4663366} studied triangle-free planar graphs without 4-cycles normally adjacent to 4- and 5-cycles.

\begin{theorem}[Han \etal \cite{MR4663366}]
Every triangle-free planar graph without $4$-cycles normally adjacent to $4$- and $5$-cycles is weakly $2$-degenerate. 
\end{theorem}

Wang \cite{MR4564473} provided the following sufficient conditions for a planar graph to be weakly $2$-degenerate.

\begin{theorem}[Wang \cite{MR4564473}]\text{}
\begin{enumerate}
\item Let $G$ be a planar graph without $4$-, $6$- and $9$-cycles. If there are no $7$-cycles normally adjacent to $5$-cycles, then $G$ is weakly $2$-degenerate. 

\item Let $G$ be a planar graph without $4$-, $6$-, and $8$-cycles. If there are no $3$-cycles normally adjacent to $9$-cycles, then $G$ is weakly $2$-degenerate. 
\end{enumerate}
\end{theorem}

In this paper, we will prove that every toroidal graph without $K_{5}^{-}$ and $6$-cycles is weakly $3$-degenerate. 

\begin{theorem}\label{WD}
Every toroidal graph without $K_{5}^{-}$ and $6$-cycles is weakly $3$-degenerate. 
\end{theorem}

\begin{corollary}
Every toroidal graph without $K_{5}^{-}$ and $6$-cycles has DP-paint number at most $4$. As a consequence, every such graph is $4$-DP-colorable. 
\end{corollary}

Next, we introduce the Alon-Tarsi number, which also has very close relationship with graph coloring parameters. A digraph $D$ is \emph{Eulerian} if $d_{D}^{+}(v) = d_{D}^{-}(v)$ for each vertex $v \in V(D)$. In particular, a digraph with no arcs is Eulerian. For a given graph $G$ and an orientation $D$ of $G$, let $\ee(D)$ be the family of spanning Eulerian sub-digraphs of $D$ with an even number of arcs, and let $\oe(D)$ be the family of spanning Eulerian sub-digraphs of $D$ with an odd number of arcs. We define the \emph{Alon-Tarsi difference} of $D$ as  
\[
\mathrm{diff}(D) = |\ee(D)| - |\oe(D)|. 
\]
We say that an orientation $D$ of $G$ is an \emph{Alon-Tarsi orientation} (AT-orientation for short) if $\mathrm{diff}(D) \neq 0$. The \emph{Alon-Tarsi number} $AT(G)$ of a graph $G$ is the smallest integer $k$ such that $G$ has an AT-orientation $D$ with $\Delta^{+}(D) < k$. The following proposition is from \cite{MR2578908}.

\begin{proposition}
For any graph $G$, we always have 
\begin{equation*}
\textsf{ch}(G) \leq \chi_{\textsf{P}}(G) \leq AT(G), 
\end{equation*}
where $\chi_{\textsf{P}}(G)$ is the paint number of $G$. 
\end{proposition}

Schauz \cite{MR2515754} proved that every planar graph has paint number at most five, and Zhu \cite{MR3906645} proved that every planar graph has Alon-Tarsi number at most five. Zhu \cite{MR4152773} also proved that every planar graph $G$ has a matching $M$ such that $AT(G - M) \leq 4$.

In \cref{sec5}, we prove that every toroidal graph without $K_{5}^{-}$ and $6$-cycles has Alon-Tarsi number at most four.

\begin{theorem}\label{AT4}
Every toroidal graph without $K_{5}^{-}$ and $6$-cycles has Alon-Tarsi number at most four. 
\end{theorem}

The paper is structured as follows. In \cref{sec2}, we present some structural results on connected toroidal graphs with $K_{5}^{-}$ and $6$-cycles. In \cref{sec3}, we prove the structural description \cref{SR} using the discharging method. In \cref{sec4}, we study weak degeneracy, and prove \cref{WD}. In \cref{sec5}, we consider the Alon-Tarsi number, and prove \cref{AT4}. Finally, we discuss the sharpness of our results in the last section. 

\section{Preliminary results}
\label{sec2}

In this section, we present some preliminary results related to the properties of a connected toroidal graph $G$ embedded on the torus.

Let $G$ be a connected toroidal graph that is embedded on the torus. A $k$-vertex is a $(d_{1}, d_{2}, \dots, d_{k})$-vertex if its incident faces have sizes $d_{i}$ in a cyclic order. We classify $4$-vertices into three types: \emph{bad}, \emph{special}, and \emph{good}. A $4$-vertex is \emph{bad} if it is a $(3, 3, 3, 3^{+})$-vertex, \emph{special} if it is a $(3, 3, 4, 6^{+})$-vertex, and \emph{good} otherwise. We use $s(f)$ to denote the number of special vertices incident with a face $f$, and $n_{4b}(v)$ to denote the number of bad vertices adjacent to a vertex $v$. A face is \emph{light} if all its incident vertices are $4$-vertices in $G$. Similar definitions can be applied to the concept of \emph{light cycles}. We say that a vertex $v$ is a $(d_{1}, d_{2}, \dots, d_{l}, \dots)$-vertex if $v$ is an $l^{+}$-vertex and consecutively incident with faces $f_{1}, f_{2}, \dots, f_{l}$ such that the size of $f_{i}$ is $d_{i}$ for all $1 \leq i \leq l$. We use $m_{i}(v)$ to denote the number of $i$-faces incident with a vertex $v$, and $n_{j}(f)$ to denote the number of $j$-vertices incident with a face $f$. We use $b(f)$ to denote the facial boundary of a face $f$. For two faces, we can define \emph{adjacent} and \emph{normally adjacent} by considering the boundaries of the two faces. 

We establish several important properties of the graph $G$. 
\begin{lemma}\label{Lem1}
Let $G$ be a connected toroidal graph without $K_{5}^{-}$ and $6$-cycles. If the minimum degree is at least 4, then the following properties hold:
\begin{enumerate}[label = (\roman*)]
\item\label{3333k} There are no $(3, 3, 3, 3, k, \dots)$-vertices for any positive integer $k$. 

\item\label{3FVF5} For a $5$-face $f= [x_{1}x_{2}x_{3}x_{4}x_{5}]$, if $f$ is adjacent to a $3$-face $[x_{1}x_{5}x]$, then $x = x_{3}$.  

\item\label{3FVF4} If $f_{1}$ and $f_{2}$ are adjacent in $G$, where $f_{1}$ is a $3$-face and $f_{2}$ is a $4$-face, then $f_{1}$ and $f_{2}$ must be normally adjacent.  

\item\label{6F} (\cite[Proposition 2.3]{MR3634481}) If $f$ is a $6$-face and $x, y, z$ are consecutive vertices on $b(f)$, then the following hold:
      \begin{enumerate}
      \item $b(f)$ consists of two triangles. 
      \item If $y$ is not incident with $f$ twice, then $xz \in E(G)$. 
      \end{enumerate}

\item\label{4FVF4} (\cite[Proposition 2.4]{MR3634481}) Let $[x_{1}x_{2}x_{3}x_{4}]$ and $[x_{1}x_{2}y_{3}y_{4}]$ be two adjacent $4$-faces. Then $|\{x_{3}, x_{4}\} \cap \{y_{3}, y_{4}\}| = 1$. Moreover, if $y_{4} \notin \{x_{1}, x_{2}, x_{3}, x_{4}\}$, then $y_{3} = x_{4}$. As a consequence, there are no normally adjacent $4$-faces. 

\item\label{343dots} (\cite[Claim 2.5]{MR3634481}) There are no $(3, 4, 3, \dots)$-vertices.

\item\label{353dots} (\cite[Claim 2.5]{MR3634481}) There are no $(3, 5, 3, \dots)$-vertices.

\item\label{3346P} (\cite[Claim 2.6]{MR3634481}) There are no $(3, 3, 4, 5^{-}, \dots)$-vertices. 

\item\label{4434k} (\cite[Claim 2.8]{MR3634481}) There are no $(4, 4, 3, 4, k)$-vertices for any positive integer $k$. 

\item\label{ONEW} (\cite[Claim 2.11]{MR3634481}) Each $4$-face is incident with at most one special $4$-vertex. 

\item\label{345P4} (\cite[Claim 2.12]{MR3634481}) If $v$ is a $(3, 4, 5^{+}, 4)$-vertex, then none of the two $4$-faces is incident with a special $4$-vertex. 

\item\label{3333} (\cite[Claim 2.16]{MR3634481}) If $v$ is a $(3, 3, 3, 3)$-vertex, then each of the four $3$-faces is incident with a $7^{+}$-face.

\item\label{3336} (\cite[Corollaries 2.19 and 2.20]{MR3634481}) If $v$ is a $(3, 3, 3, k)$-vertex with $k \geq 4$, then $k \geq 6$ and each edge not incident with $v$ but on the $3$-faces is incident with a $6^{+}$-face. Moreover, if $k = 6$, then each edge not incident with $v$ but on the $3$-faces is incident with a $7^{+}$-face. 

\item\label{NABAD} (\cite[Corollary 2.22]{MR3634481}) There are no adjacent bad vertices.

\item\label{444dots} There are no $(4, 4, 4, \dots)$-vertices.

\item\label{3k4l} Let $v$ be a $(3, k, 4, l)$-vertex with $k, l \geq 4$. Then 
      \begin{enumerate}
      \item $k, l \neq 5$ and $\max\{k, l\} \geq 6$; and 
      \item none of the two $4$-faces can be incident with a special $4$-vertex if $k = 4$. 
      \end{enumerate}
\item\label{335k} (\cite[Claim 2.14]{MR3634481}) Let $v$ be a $(3, 3, 5, k)$-vertex with $k \geq 5$. Then $k \geq 7$. Furthermore, if there is no induced subgraph isomorphic to the configuration in \cref{K5--}, then the $5$-face is incident with a $5^{+}$-vertex. 

\item\label{3454}
Let $v$ be a $(3, 4, 5, 4)$-vertex incident with faces $f_{1} = [v_{1}vv_{2}]$, $f_{2} = [v_{2}vv_{3}x]$, $f_{3} = [v_{3}vv_{4}pq]$ and $f_{4} = [v_{4}vv_{1}y]$. If $v_{3}$ is a $4$-vertex, then $v_{3}$ is a $(4^{+}, 4^{+}, 4^{+}, 4^{+})$-vertex. 

\item\label{m64b}
For each vertex $v$, we always have $m_{6^{+}}(v) \geq n_{4b}(v)$. 

\end{enumerate}
\end{lemma}


\begin{proof}[Proof of \cref{Lem1}\ref{444dots}]
Assume that a vertex $v$ is incident with three consecutive $4$-faces, say $f_{1} = [vv_{1}xv_{2}]$, $f_{2} = [vv_{2}yv_{3}]$, and $f_{3} = [vv_{3}zv_{4}]$. Consider the first case that $y \notin \{v_{1}, v_{4}\}$. Since $f_{1}$ and $f_{2}$ are adjacent $4$-faces, we have $x = v_{3}$ by \cref{Lem1}\ref{4FVF4}. Similarly, we have $z = v_{2}$. In this case, there is a $6$-cycle $vv_{4}v_{2}yv_{3}v_{1}v$, a contradiction. The other case is that $y \in \{v_{1}, v_{4}\}$. By symmetry, we may assume that $y = v_{1}$. Note that \cref{Lem1}\ref{4FVF4} implies that $x \neq v_{3}$. Similarly, we have $z = v_{2}$ since $f_{2}$ and $f_{3}$ are two adjacent $4$-faces. If $x \neq v_{4}$, then there is a $6$-cycle $vv_{4}v_{2}xv_{1}v_{3}v$, a contradiction. If $x = v_{4}$, then $G[\{v, v_{1}, v_{2}, v_{3}, v_{4}\}]$ contains a $K_{5}^{-}$, a contradiction. This completes the proof of \cref{Lem1}\ref{444dots}.
\end{proof}

\begin{proof}[Proof of \cref{Lem1}\ref{3k4l}]
Let $v_{1}, v_{2}, v_{3}, v_{4}$ be the four neighbors of $v$ in a cyclic order, and let $f_{1} = [v_{1}vv_{2}]$ be a $3$-face, and $f_{3} = [v_{3}vv_{4}x]$ be a $4$-face. Let $f_{2} = [v_{2}vv_{3}\dots]$ and $f_{4} = [v_{4}vv_{1}\dots]$. 
      
      Assume that $f_{4} = [v_{4}vv_{1}zw]$ is a $5$-face. By \cref{Lem1}\ref{3FVF5}, $f_{1}$ and $f_{4}$ are not normally adjacent, and $w = v_{2}$. If $x \notin \{v_{1}, v_{2}\}$, then there is a $6$-cycle $v_{1}v_{2}v_{4}xv_{3}vv_{1}$, a contradiction. Then we have $x \in \{v_{1}, v_{2}\}$. Since $v_{4}$ is a $4^{+}$-vertex, we have $x \neq w$. It follows that $x \neq v_{2}$, implying $x = v_{1}$. If $z \neq v_{3}$, then there is a $6$-cycle $vv_{3}v_{1}zv_{2}v_{4}v$, a contradiction. This implies that $z = v_{3}$, but $G[\{v, v_{1}, v_{2}, v_{3}, v_{4}\}]$ contains a $K_{5}^{-}$, a contradiction. Therefore, by symmetry, neither $f_{2}$ nor $f_{4}$ can be a $5$-face. By \cref{Lem1}\ref{444dots}, at most one of $f_{2}$ and $f_{4}$ can be a $4$-face. Hence, $\max\{k, l\} \geq 6$. 
      
      In what follows, we assume that $f_{2} = [v_{2}vv_{3}y]$ is a $4$-face. Assume that $x \notin \{v_{1}, v_{2}\}$. By \cref{Lem1}\ref{4FVF4}, $f_{2}$ and $f_{3}$ are not normally adjacent, and $y = v_{4}$. Then there is a $6$-cycle $v_{1}vv_{3}xv_{4}v_{2}v_{1}$, a contradiction. So we may assume that $x \in \{v_{1}, v_{2}\}$. Note that the $3$-face $f_{1}$ and the $4$-face $f_{2}$ are normally adjacent. Suppose that $x = v_{1}$. It follows that $x \neq v_{2}$. By \cref{Lem1}\ref{4FVF4}, $f_{2}$ and $f_{3}$ are not normally adjacent, and $y = v_{4}$, but $G[\{v, v_{1}, v_{2}, v_{3}, v_{4}\}]$ contains a $K_{5}^{-}$, a contradiction. This implies that $x = v_{2}$, and hence $y \neq v_{4}$ by \cref{Lem1}\ref{4FVF4}. 
      
      By \cref{Lem1}\ref{343dots}, $v_{2}y$ cannot be incident with a $3$-face. Suppose that $v_{3}y$ is incident with a $3$-face $[v_{3}yz]$. If $z \neq v_{1}$, then there is a $6$-cycle $vv_{1}v_{2}yzv_{3}v$; otherwise, $z = v_{1}$ and $G[\{v, v_{1}, v_{2}, v_{3}, y\}]$ contains a $K_{5}^{-}$. Then $v_{3}y$ cannot be incident with a $3$-face. Since none of $v_{2}y$ and $v_{3}y$ is incident with a $3$-face, we have that $y$ is not a special vertex. Note that $x = v_{2}$, each of $x$ and $v_{2}$ is incident with two $4$-faces $f_{2}$ and $f_{3}$, thus none of $x$ and $v_{2}$ is a special vertex. Similarly, each of $v$ and $v_{3}$ is incident with two $4$-faces $f_{2}$ and $f_{3}$, thus none of $v$ and $v_{3}$ is a special vertex. Finally, we consider the vertex $v_{4}$. Suppose that $v_{4}x$ is incident with a $3$-face $[v_{4}xt]$. If $t \notin \{y, v_{3}\}$, then there is a $6$-cycle $vv_{3}yv_{2}tv_{4}v$, a contradiction. If $t = v_{3}$, then $x$ is a $2$-vertex, a contradiction. If $t = y$, then there is a $6$-cycle $vv_{1}v_{2}v_{4}yv_{3}v$, a contradiction. Then $v_{4}x$ is not incident with a $3$-face, and $v_{4}$ is incident with at most one $3$-face, so $v_{4}$ is not a special vertex. Therefore, none of $f_{2}$ and $f_{3}$ is incident with a special vertex. This completes the proof of \cref{Lem1}\ref{3k4l}.
\end{proof}

\begin{proof}[Proof of \cref{Lem1}\ref{335k}]
Let $v_{1}, v_{2}, v_{3}, v_{4}$ be the four neighbors of $v$ in a cyclic order so that $f_{1} = [v_{1}vv_{2}]$, $f_{2} = [v_{2}vv_{3}]$ be two $3$-faces, $f_{3} = [v_{3}vv_{4}xy]$ be a $5$-face, and $f_{4} = [v_{4}vv_{1}\dots]$ be a $5^{+}$-face. By \cref{Lem1}\ref{3FVF5}, we have that $x = v_{2}$.

Suppose that $f_{4} = [v_{4}vv_{1}zw]$ is a $5$-face. Consider the $3$-face $f_{1}$ and the $5$-face $f_{4}$, by \cref{Lem1}\ref{3FVF5}, we have that $w = v_{2}$. Then $x = w = v_{2}$, the vertex $v_{4}$ must be a $2$-vertex, a contradiction. Suppose that $f_{4}$ is a $6$-face. By \cref{Lem1}\ref{6F}, $b(f_{4})$ consists of two triangles. Note that $v$ cannot be incident with $f_{4}$ twice, then $v_{1}v_{4}$ is an edge. Note that $v_{1} \neq x$ since $x = v_{2}$. If $v_{1} \neq y$, then there is a $6$-cycle $vv_{1}v_{4}xyv_{3}v$, a contradiction. If $v_{1} = y$, then $G[\{v, v_{4}, x, y, v_{3}\}]$ contains a $K_{5}^{-}$, a contradiction. It follows that $f_{4}$ is a $7^{+}$-face, and $k \geq 7$. 

If $v_{1} \neq y$, then $x$ is adjacent to each of $y, v_{3}, v, v_{4}, v_{1}$, and then $d(x) \geq 5$. So we may assume that $v_{1} = y$. Then $G[\{v_{3}, v, v_{4}, x, y\}]$ contains a subgraph isomorphic to \cref{GI}. Note that neither $v_{3}v_{4}$ nor $v_{4}y$ is an edge, otherwise there is a $K_{5}^{-}$, a contradiction. Hence, $G[\{v_{3}, v, v_{4}, x, y\}]$ is isomorphic to \cref{GI}. Since there is no induced subgraph isomorphic to \cref{K5--}, the five vertices on $b(f_{3})$ cannot be all $4$-vertices, thus $f_{3}$ is incident with a $5^{+}$-vertex. This completes the proof of \cref{Lem1}\ref{335k}.
\end{proof}

\begin{proof}[Proof of \cref{Lem1}\ref{3454}]
Let $t$ be the fourth neighbor of $v_{3}$. Suppose that $v_{3}x$ is incident with a $3$-face. Then $t = v_{1}$, otherwise $vv_{1}v_{2}xtv_{3}v$ is a $6$-cycle. Since $t \neq q$ and $t = v_{1}$, we have that $p = v_{1}$, otherwise there is a $6$-cycle $vv_{1}v_{3}qpv_{4}v$. If $x \neq v_{4}$, then there is a $6$-cycle $vv_{4}v_{1}v_{2}xv_{3}v$. Then $x = v_{4}$, but now $G[\{v, v_{1}, v_{2}, v_{3}, v_{4}\}]$ contains a $K_{5}^{-}$, a contradiction. Hence, $v_{3}x$ is incident with two $4^{+}$-faces.

Suppose that $v_{3}q$ is incident with a $3$-face. By \cref{Lem1}\ref{3FVF5}, we have that $t = v_{4}$. Now, there is a $6$-cycle $vv_{1}v_{2}xv_{3}v_{4}v$, a contradiction. Hence, $v_{3}q$ is incident with two $4^{+}$-faces. Therefore, $v_{3}$ is incident with four $4^{+}$-faces. This completes the proof of \cref{Lem1}\ref{3454}.
\end{proof}

\begin{figure}%
\centering
\begin{tikzpicture}[line width = 1pt]
\def\s{1}
\foreach \ang in {1, 2, 3, 4, 5}
{
\def\pointname{v\ang}
\coordinate (\pointname) at ($(\ang*360/5-54:\s)$);}
\draw (v1)node[right]{$x$}--(v2)node[above]{$v_{4}$}--(v3)node[left]{$v$}--(v4)node[left]{$v_{3}$}--(v5)node[right]{$y$}--cycle;
\draw (v5)--(v3)--(v1)--(v4);
\foreach \ang in {1, 2, 3, 4, 5}
{
\node[circle, inner sep = 1.5, fill = white, draw] () at (v\ang) {};
}
\end{tikzpicture}
\caption{A subgraph.}
\label{GI}
\end{figure}
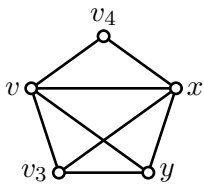

\begin{proof}[Proof of \cref{Lem1}\ref{m64b}]
Let $v_{1}, v_{2}, \dots, v_{d(v)}$ be the neighbors of $v$ in a cyclic order, and let $f_{i}$ be the face incident with $vv_{i}$ and $vv_{i+1}$ for $1 \leq i \leq d(v)$. Clearly, the inequality holds for $n_{4b}(v) = 0$. Let $v_{i}$ be a bad vertex. By the definition, $vv_{i}$ is incident with at least one $3$-face. 

The first case is that $vv_{i}$ is incident with exactly one $3$-face, say $f_{i}$. Then $v_{i}v_{i+1}$ is an edge. By \cref{Lem1}\ref{3336}, each of $f_{i-1}$ and $f_{i+1}$ is a $6^{+}$-face, and at least one of them is a $7^{+}$-face. \cref{Lem1}\ref{NABAD} implies that exactly one of $v_{i}$ and $v_{i+1}$, say $v_{i}$, is a bad vertex.

Now, consider the second case that $vv_{i}$ is incident with two $3$-faces. Then $v_{i-1}v_{i}$ and $v_{i}v_{i+1}$ are edges. By \cref{Lem1}\ref{3333}, each of $f_{i-2}$ and $f_{i+1}$ is a $6^{+}$-face. \cref{Lem1}\ref{NABAD} implies that exactly one of $v_{i-1}, v_{i}$ and $v_{i+1}$, say $v_{i}$, is a bad vertex.

By the above discussion, we can obtain that the number of incident $6^{+}$-faces is not less than the number of adjacent bad vertices. This completes the proof of \cref{Lem1}\ref{m64b}.
\end{proof}

\section{Proof of \cref{SR}}
\label{sec3}
In this section, we prove the structural description presented in \cref{SR}. Assume that $G$ is a counterexample to \cref{SR}. This means that none of the structures described in \cref{SR} exist in $G$. We will use the discharging method to derive a contradiction. 

Before proceeding with the discharging process, we will introduce the following lemma.

\begin{lemma}\label{3CVC4}
There is no subgraph (not necessarily induced) isomorphic to the configuration shown in \cref{House}. 
\end{lemma}
\begin{proof}
Suppose to the contrary that there is a subgraph isomorphic to the configuration shown in \cref{House}. Assume that $x_{1}x_{2}x_{3}x_{4}x_{1}$ is the $4$-cycle, $x_{1}x_{4}x_{5}x_{1}$ is the $3$-cycle, and all the five vertices are $4$-vertices in $G$. 
Since there is no induced subgraph isomorphic to the configuration shown in \cref{Kite}, we have that $x_{1}x_{3}, x_{2}x_{4} \notin E(G)$ or $x_{1}x_{3}, x_{2}x_{4} \in E(G)$. If $x_{1}x_{3}, x_{2}x_{4} \in E(G)$, then $G[\{x_{1}, x_{2}, x_{3}, x_{4}, x_{5}\}]$ is isomorphic to the configuration shown  in \cref{K5--} or it contains a $K_{5}^{-}$, a contradiction. If $x_{1}x_{3}, x_{2}x_{4} \notin E(G)$, then $G[\{x_{1}, x_{2}, x_{3}, x_{4}, x_{5}\}]$ is isomorphic to the configuration shown in \cref{House} or it contains a configuration shown in \cref{Kite}, both of which lead to contradictions. 
\end{proof}

We assign an initial charge $\omega(v) = d(v) - 6$ to each vertex $v \in V(G)$, and an initial charge $\omega(f) = 2d(f) - 6$ to each face $f \in F(G)$. By Euler's formula and Handshaking lemma, we can see that the sum of the initial charges is zero, \ie
\[
\sum_{x \in V(G) \cup F(G)} \omega(x) = \sum_{v \in V(G)} (d(v) - 6) + \sum_{f \in F(G)} (2d(f) - 6) = 0
\]

In the following, we design some appropriate discharging rules to redistribute the charges, obtaining a final charging function $\omega'$ on $V(G) \cup F(G)$, such that $\omega'(x) \geq 0$ for each $x \in V(G) \cup F(G)$. Moreover, there exists a vertex such that its final charge is positive. Note that the sum of the charges is preserved in the discharging procedure, then 
\[
0 = \quad \sum_{x \in V(G) \cup F(G)} \omega(x) = \sum_{x \in V(G) \cup F(G)} \omega'(x) \quad > 0,
\]
this contradiction completes the proof of \cref{SR}. 

Let $\epsilon$ and $\eta$ be sufficiently small positive real numbers, and let $0 < \lambda \leq \rho < 1$. We use the following discharging rules. 
\begin{enumerate}[label = \textbf{R\arabic*}, ref = R\arabic*]
\item\label{R1} Let $f$ be a $4$-face incident with a vertex $v$.
     \begin{enumerate}[label = \textbf{R1.\arabic*}]
     \item $\tau(f \rightarrow v) = \eta$ if $v$ is a $5^{+}$-vertex.
     \item Let $v$ be a $4$-vertex.
          \begin{enumerate}[label = \textbf{(\roman*)}]
          \item $\tau(f \rightarrow v) = 1$ if $v$ is special
          \item $\tau(f \rightarrow v) = \frac{2 - s(f) - \eta \times n_{5^{+}}(f)}{n_{4}(f) - s(f)}$ otherwise. 
          \end{enumerate}
     \end{enumerate}
\item\label{R2} Let $f$ be a $5$-face incident with a vertex $v$.
     \begin{enumerate}
     \item $\tau(f \rightarrow v) = \frac{1}{2}$ if $d(v) \geq 5$
     \item $\tau(f \rightarrow v) = \frac{4 - \frac{1}{2}n_{5^{+}}(f)}{n_{4}(f)}$ if $d(v) = 4$. 
     \end{enumerate}
\item\label{R3} Every $6^{+}$-face $f$ sends $\frac{2d(f) - 6}{d(f)}$ to each incident vertex. 
\item\label{R7} Every $(4^{+}, 4^{+}, 4^{+}, 4^{+})$-vertex sends $\frac{1}{20}$ to each adjacent vertex. 
\item\label{R4} Every good $4$-vertex sends its remaining charge uniformly to each adjacent bad $4$-vertex.
\item\label{R5} Every $5$-vertex sends $\lambda = \frac{5}{7} + \frac{\epsilon}{2}$ to each adjacent $(3, 3, 3, 3)$-vertex, and $\mu = \frac{4}{7} - \epsilon$ to each adjacent $(3, 3, 3, 4^{+})$-vertex. 
\item\label{R6} Every $6^{+}$-vertex sends $\rho$ to each adjacent bad $4$-vertex. 
\end{enumerate}

Firstly, we show that $\omega'(x) \geq 0$ for all $x \in V(G) \cup F(G)$. Note that $\omega'(f) = \omega(f) = 2 \times 3 - 6 = 0$ for each $3$-face $f$. By \cref{Lem1}\ref{ONEW}, every $4$-face is incident with at most one special $4$-vertex, then its final charge is at least zero by \ref{R1}. By \ref{R2} and \ref{R3}, each $5^{+}$-face has a nonnegative final charge.   

Let $v$ be a $6^{+}$-vertex. Note that $v$ is incident with at least one $4^{+}$-face by \cref{Lem1}\ref{3333k}. If $v$ is not adjacent to a bad $4$-vertex, then it does not send out any charge, but it receives at least $\eta$ from each incident $4^{+}$-face, thus $\omega'(v) \geq d(v) - 6 + \eta \geq \eta$. Now, assume that $v$ is adjacent to a bad $4$-vertex. By \cref{Lem1}\ref{m64b}, we have that $m_{6^{+}}(v) \geq n_{4b}(v)$. By \ref{R3} and \ref{R6}, $v$ receives at least $1$ from each incident $6^{+}$-face, and sends $\rho$ to each adjacent bad $4$-vertex. Then $\omega'(v) \geq d(v) - 6 + m_{6^{+}}(v) - \rho \times n_{4b}(v) > 0$. 

Let $v$ be a $5$-vertex. Then $\omega(v) = 5 - 6 = -1$. Suppose that $v$ is adjacent to at least three bad $4$-vertices. Then there are two adjacent bad $4$-vertices $v_{i}$ and $v_{i+1}$, say $v_{1}$ and $v_{2}$. By \cref{Lem1}\ref{NABAD}, $f_{1}$ is a $4^{+}$-face. This implies that both $f_{2}$ and $f_{5}$ are $3$-faces. By \cref{Lem1}\ref{NABAD}, neither $v_{3}$ nor $v_{5}$ is a bad $4$-vertex. \cref{Lem1}\ref{3336} implies that both $f_{3}$ and $f_{4}$ are $6^{+}$-faces. Then $v_{4}$ is not a bad vertex, and $v$ is adjacent to exactly two bad $4$-vertices, a contradiction. Hence, $v$ is adjacent to at most two bad $4$-vertices. By \cref{Lem1}\ref{3333}, $v$ is adjacent to at most one $(3, 3, 3, 3)$-vertex. Suppose that $v$ is adjacent to a $(3, 3, 3, 3)$-vertex, say $v_{1}$. Then $f_{1}$ and $f_{5}$ are $3$-faces. \cref{Lem1}\ref{3333} implies that both $f_{2}$ and $f_{4}$ are $7^{+}$-faces. By \ref{R3}, each of $f_{2}$ and $f_{4}$ sends at least $\frac{8}{7}$ to $v$. It follows that $\omega'(v) \geq -1 + \frac{8}{7} \times 2 - \lambda - \mu = \frac{\epsilon}{2} > 0$. So we may assume that $v$ is not adjacent to a $(3, 3, 3, 3)$-vertex, and it is adjacent to at most two $(3, 3, 3, 4^{+})$-vertices. 

\begin{itemize}
\item $\bm{n_{4b}(v) \geq 1}$. Then $v$ is incident with at least two $6^{+}$-faces by \cref{Lem1}\ref{3336}. By \ref{R3}, each incident $6$-face sends $1$ to $v$, and each incident $7^{+}$-face sends at least $\frac{8}{7}$ to $v$. If $v$ is adjacent to exactly one $(3, 3, 3, 4^{+})$-vertex, then $\omega'(v) \geq -1 + 1 \times 2 - \mu > 0$. Suppose that $v$ is adjacent to exactly two $(3, 3, 3, 4^{+})$-vertices. Note that $v$ is a $5$-vertex in $G$. By \cref{Lem1}\ref{3336}, there is a bad $4$-vertex $v_{i}$, say $v_{1}$, such that $vv_{1}$ is incident with exactly one $3$-face. Without loss of generality, let $f_{1}$ be a $3$-face. By \cref{Lem1}\ref{3336}, at least one of $f_{2}$ and $f_{5}$ is a $7^{+}$-face. It follows that $\omega'(v) \geq -1 + 1 + \frac{8}{7} - 2\mu = 2\epsilon > 0$. 

\item $\bm{n_{4b}(v) = 0}$. Note that $v$ is incident with at most three $3$-faces, otherwise there is a $6$-cycle in $G$. Then $v$ is incident with at least two $4^{+}$-faces. If $v$ is incident with a $6^{+}$-face, then it receives at least $1$ from an incident $6^{+}$-face, and at least $\eta$ from another $4^{+}$-face, this implies that $\omega'(v) \geq -1 + \eta + 1 > 0$. So we may assume that $v$ is not incident with a $6^{+}$-face. Then all the incident $4^{+}$-faces are $4$- or $5$-faces. By \ref{R1} and \ref{R2}, each incident $4$-face sends $\eta$ to $v$, and each incident $5$-face sends $\frac{1}{2}$ to $v$. By \cref{Lem1}\ref{343dots} and \ref{353dots}, all the incident $3$-faces are consecutive. 
     \begin{itemize}
     \item $\bm{m_{3}(v) = 3}$. Since all the incident $3$-faces are consecutive, we may assume that $f_{1}, f_{2}$ and $f_{3}$ are $3$-faces. By \cref{Lem1}\ref{3346P}, both $f_{4}$ and $f_{5}$ are $5$-faces. Let $f_{4} = [v_{4}vv_{5}xy]$ and $f_{5} = [v_{5}vv_{1}st]$ be the two $5$-faces. By \cref{Lem1}\ref{3FVF5}, we have that $x = v_{3}$ and $t = v_{2}$. Then there is a $6$-cycle $vv_{1}v_{2}v_{5}v_{3}v_{4}v$, a contradiction. 
     
     \item $\bm{m_{3}(v) = 2}$. Since all the incident $3$-faces are consecutive, we may assume that $f_{1}$ and $f_{2}$ are $3$-faces. By \cref{Lem1}\ref{3346P}, both $f_{3}$ and $f_{5}$ are $5$-faces. By \ref{R1} and \ref{R2}, each of $f_{3}$ and $f_{5}$ sends $\frac{1}{2}$ to $v$, and $f_{4}$ sends at least $\eta$ to $v$. Then $\omega'(v) \geq -1 + \frac{1}{2} \times 2 + \eta > 0$. 
     
     \item $\bm{m_{3}(v) = 1}$. Let $f_{1}$ be a $3$-face. If $v$ is incident with at least two $5$-faces, then $\omega'(v) \geq -1 + \frac{1}{2} \times 2 + 2\eta > 0$. Suppose that $v$ is incident with at most one $5$-face. By \cref{Lem1}\ref{444dots}, $f_{3}$ or $f_{4}$ is a $5$-face, and the other three faces are $4$-faces. But this contradicts \cref{Lem1}\ref{4434k}. 
     
     \item $\bm{m_{3}(v) = 0}$. By \cref{Lem1}\ref{444dots}, $v$ is incident with at most three $4$-faces, then it is incident with at least two $5$-faces. Then $\omega'(v) \geq -1 + \frac{1}{2} \times 2 + 3\eta > 0$. 
     \end{itemize}

\end{itemize}

Let $v$ be a $4$-vertex. Then $\omega(v) = 4 - 6 = -2$. If $v$ is a $(3, 3, 4, 6^{+})$-vertex, then the incident $4$-face sends $1$ to this special $4$-vertex $v$, and the incident $6^{+}$-face sends at least $1$ to $v$, thus 
\begin{equation}\label{EQ1}
\omega'(v) \geq -2 + 1 + 1 = 0. 
\end{equation}

Assume that $v$ is a $(3, 3, 3, 3)$-vertex. Note that $v_{1}v_{3}, v_{2}v_{4} \notin E(G)$, otherwise $G[\{v, v_{1}, v_{2}, v_{3}, v_{4}\}]$ contains a $K_{5}^{-}$. Since there is no induced subgraph isomorphic to the configuration in \cref{Kite}, $v$ is adjacent to at least two $5^{+}$-vertices. By \cref{Lem1}\ref{3333}, each of $v_{1}v_{2}, v_{2}v_{3}, v_{3}v_{4}$ and $v_{4}v_{1}$ is incident with a $7^{+}$-face. Then each $4$-vertex in $\{v_{1}, v_{2}, v_{3}, v_{4}\}$ is a good $4$-vertex. Moreover, each $4$-vertex in $\{v_{1}, v_{2}, v_{3}, v_{4}\}$ is adjacent to exactly one bad $4$-vertex, say $v$. By \ref{R4}, each adjacent good $4$-vertex sends at least $-2 + \frac{8}{7} \times 2 = \frac{2}{7}$ to $v$. Then $\omega'(v) \geq -2 + 2\lambda + \frac{2}{7} \times 2 = \epsilon > 0$. 

Assume that $v$ is a $(3, 3, 3, 4^{+})$-vertex. Let $f_{1}, f_{2}$ and $f_{3}$ be $3$-faces. By \cref{Lem1}\ref{3336}, $f_{4}$ is a $6^{+}$-face, and each of $v_{1}v_{2}, v_{2}v_{3}, v_{3}v_{4}$ is incident with a $6^{+}$-face. Then each of $v_{1}, v_{2}, v_{3}$ and $v_{4}$ is incident with at least two $6^{+}$-faces. By definition, each $4$-vertex in $\{v_{1}, v_{2}, v_{3}, v_{4}\}$ is a good $4$-vertex, each $4$-vertex in $\{v_{2}, v_{3}\}$ is adjacent to exactly one bad $4$-vertex, say $v$, and each $4$-vertex in $\{v_{1}, v_{4}\}$ is adjacent to at most two bad $4$-vertices. By \cref{Lem1}\ref{3336}, each of $v_{1}$ and $v_{4}$ is incident with a $7^{+}$-face. By \ref{R3} and \ref{R4}, each $4$-vertex in $\{v_{1}, v_{4}\}$ sends at least $(-2 + 1 + \frac{8}{7})/2 = \frac{1}{14}$ to $v$. By \ref{R5} and \ref{R6}, each $5^{+}$-vertex in $\{v_{1}, v_{2}, v_{3}, v_{4}\}$ sends at least $\mu$ to $v$. 
    \begin{itemize}
    \item \textbf{Assume that $f_{4}$ is a $6$-face}. Since $v$ cannot be incident with $f_{4}$ twice, we have that $v_{1}v_{4}$ is an edge in $G$. Moreover, $v_{1}$ or $v_{4}$ is incident with $f_{4}$ twice. It follows that $v_{1}$ or $v_{4}$, say $v_{1}$, is a $5^{+}$-vertex. Similar to the case that $v$ is a $(3, 3, 3, 3)$-vertex, we have that $v_{1}v_{3}, v_{2}v_{4} \notin E(G)$, and $v$ is adjacent to at least two $5^{+}$-vertices. By \cref{Lem1}\ref{3336}, each of $v_{1}v_{2}, v_{2}v_{3}$ and $v_{3}v_{4}$ is incident with a $7^{+}$-face. By \ref{R3} and \ref{R4}, each $4$-vertex in $\{v_{2}, v_{3}\}$ sends at least $-2 + \frac{8}{7} \times 2 = \frac{2}{7}$ to $v$. Therefore, $\omega'(v) \geq -2 + 1 + 2\mu + \frac{2}{7} + \frac{1}{14} = \frac{1}{2} - 2\epsilon > 0$. 
    
    \item \textbf{Assume that $f_{4}$ is a $7^{+}$-face}. By \ref{R3}, $f_{4}$ sends at least $\frac{8}{7}$ to $v$. If $v$ is adjacent to at least two $5^{+}$-vertices, then $\omega'(v) \geq -2 + \frac{8}{7} + 2\mu = \frac{2}{7} - 2\epsilon > 0$. So we may assume that $v$ is adjacent to at most one $5^{+}$-vertex. Since $G[\{v, v_{1}, v_{2}, v_{3}, v_{4}\}]$ contains no $K_{5}^{-}$, we have that $v_{1}v_{3} \notin E(G)$ or $v_{2}v_{4} \notin E(G)$. 
        \begin{itemize}
         \item \textbf{Suppose that $\bm{d(v_{2}) = d(v_{3}) = 4}$}. Without loss of generality, assume that $v_{2}v_{4} \notin E(G)$. Since there is no induced subgraph isomorphic to the configuration in \cref{Kite}, we have that $v_{4}$ is a $5^{+}$-vertex. It follows that $v_{1}$ is a $4$-vertex. Since $G[\{v, v_{1}, v_{2}, v_{3}\}]$ is not isomorphic to the configuration in \cref{Kite}, then $v_{1}v_{3}$ is an edge in $G$. Suppose that $v_{1}v_{2}$ is incident with a $6$-face $[v_{1}v_{2}u\dots v_{1}]$. Since $v_{2}$ is a $4$-vertex, we have that $v_{2}$ is incident with the $6$-face once. Then there is a $6$-cycle $vv_{1}uv_{2}v_{3}v_{4}v$, or $u = v_{4}$ and $v_{2}v_{4} \in E(G)$, a contradiction. Hence, $v_{1}v_{2}$ is incident with a $7^{+}$-face. Similarly, we can prove that each of $v_{2}v_{3}$ and $v_{3}v_{4}$ is incident with a $7^{+}$-face. Then each of $v_{1}, v_{2}, v_{3}$ and $v_{4}$ is incident with two $7^{+}$-faces. Moreover, each of $v_{2}$ and $v_{3}$ is a good $4$-vertex adjacent to exactly one bad $4$-vertex, say $v$. By \ref{R3} and \ref{R4}, each of $v_{2}$ and $v_{3}$ sends at least $-2 + \frac{8}{7} \times 2 = \frac{2}{7}$ to $v$. Then $\omega'(v) \geq -2 + \frac{8}{7} + \mu + \frac{2}{7} \times 2 > 0$. 
         \item \textbf{Suppose that one of $v_{2}$ and $v_{3}$, say $v_{2}$, is a $5^{+}$-vertex}. If $v_{1}v_{3}, v_{1}v_{4} \in E(G)$, then $G[\{v, v_{1}, v_{2}, v_{3}, v_{4}\}]$ contains a $K_{5}^{-}$, a contradiction. If exactly one of $v_{1}v_{3}$ and $v_{1}v_{4}$ is an edge, then there is an induced subgraph isomorphic to the configuration in \cref{Kite}, a contradiction. It follows that $v_{1}v_{3}, v_{1}v_{4} \notin E(G)$. Similar to the above cases, we can prove that each of $v_{1}v_{2}, v_{2}v_{3}$ and $v_{3}v_{4}$ is incident with a $7^{+}$-face. Moreover, $v_{3}$ is adjacent to exactly one bad $4$-vertex, and each of $v_{1}$ and $v_{4}$ is adjacent to at most two bad $4$-vertices. By \ref{R3} and \ref{R4}, $v_{3}$ sends at least $-2 + \frac{8}{7} \times 2 = \frac{2}{7}$ to $v$, and each of $v_{1}$ and $v_{4}$ sends at least $(-2 + \frac{8}{7} \times 2)/2 = \frac{1}{7}$ to $v$. Then $\omega'(v) \geq -2 + \frac{8}{7} + \mu + \frac{2}{7} + \frac{1}{7} \times 2 = \frac{2}{7} - \epsilon > 0$. 
         \end{itemize}
    \end{itemize}
    
In what follows, we assume that $v$ is a good vertex. Then $v$ is incident with at most two $3$-faces. By \cref{Lem1}\ref{ONEW}, each $4$-face is incident with at most one special $4$-vertex. By \ref{R1}, each incident $4$-face sends at least $\frac{1}{3}$ to $v$. By \ref{R2}, each incident $5$-face sends at least $\frac{4}{5}$ to $v$. By \ref{R3}, each incident $6^{+}$-face sends at least $1$ to $v$. Let $\omega^{*}(v)$ be the charge of $v$ by only applying \ref{R1}, \ref{R2}, \ref{R3} and \ref{R7}. By \ref{R4}, it suffices to prove $\omega^{*}(v) \geq 0$.

\begin{itemize}
\item $\bm{m_{3}(v) = 0}$. By \cref{Lem1}\ref{444dots}, $v$ is incident with at least two $5^{+}$-faces. By \ref{R2} and \ref{R3}, each incident $5^{+}$-face sends at least $\frac{4}{5}$ to $v$. Recall that each incident $4$-face sends at least $\frac{1}{3}$ to $v$. Obviously, $v$ is not adjacent to any bad $4$-vertex. Then $\omega^{*}(v) \geq -2 + \frac{4}{5} \times 2 + \frac{1}{3} \times 2 - \frac{1}{20} \times 4 = \frac{1}{15} > 0$. 

\item $\bm{m_{3}(v) = 2}$. If the two $3$-faces are not adjacent in $G$, then $v$ is incident with two $6^{+}$-faces by \cref{Lem1}\ref{343dots} and \ref{353dots}, thus 
\begin{equation}\label{EQ2}
\omega^{*}(v) \geq -2 + 1 \times 2 = 0.
\end{equation} Assume that $v$ is incident with two adjacent $3$-faces, say $f_{1}$ and $f_{2}$. Note that $v$ is not a special vertex. By \cref{Lem1}\ref{3346P}, $v$ is incident with two $5^{+}$-faces. Suppose that $f_{3}$ is a $5$-face. By \cref{Lem1}\ref{335k}, $f_{4}$ is a $7^{+}$-face, and $f_{3}$ is incident with at least one $5^{+}$-vertex. By \ref{R2}, the $5$-face $f_{3}$ sends at least $\frac{4 - \frac{1}{2}}{4} = \frac{7}{8}$ to $v$. Hence, $\omega^{*}(v) \geq -2 + \frac{8}{7} + \frac{7}{8} = \frac{1}{56} > 0$. By symmetry, we may assume that both $f_{3}$ and $f_{4}$ are $6^{+}$-faces. Then 
\begin{equation}\label{EQ3}
\omega^{*}(v) \geq -2 + 1 \times 2 = 0.
\end{equation}

\item $\bm{m_{3}(v) = 1}$. If $v$ is incident with three $5^{+}$-faces, then it receives at least $\frac{4}{5}$ from each incident $5^{+}$-face, and then $\omega^{*}(v) \geq -2 + \frac{4}{5} \times 3 = \frac{2}{5} > 0$. If $v$ is incident with at least two $6^{+}$-faces, then $\omega^{*}(v) \geq -2 + 1 \times 2 + \frac{1}{3} > 0$. So we may assume that $v$ is incident with a $4$-face and at most one $6^{+}$-face. Without loss of generality, let $f_{1}$ be a $3$-face. 
    \begin{itemize}
    \item Assume that $v$ is a $(3, 4, 5^{+}, 4)$-vertex. By \cref{Lem1}\ref{345P4}, none of $f_{2}$ and $f_{4}$ is incident with a special $4$-vertex. If each of $f_{2}$ and $f_{4}$ sends at least $\frac{2 - \eta}{3}$ to $v$, then $\omega^{*}(v) \geq -2 + \frac{2 - \eta}{3} \times 2 + \frac{4}{5} = \frac{2 - 10\eta}{15} > 0$. So we may assume that $f_{2}$ or $f_{4}$, say $f_{2}$, sends less than $\frac{2 - \eta}{3}$ to $v$. By \ref{R1}, $f_{2}$ is incident with four $4$-vertices, and it sends $\frac{1}{2}$ to $v$. By \cref{Lem1}\ref{3FVF4}, $f_{1}$ and $f_{2}$ are normally adjacent. By \cref{3CVC4}, we have $d(v_{1}) \geq 5$. By \ref{R1}, $f_{4}$ sends at least $\frac{2 - \eta}{3}$ to $v$. If $f_{3}$ is a $6^{+}$-face, then $\omega^{*}(v) \geq -2 + \frac{1}{2} + 1 + \frac{2 - \eta}{3} = \frac{1 - 2 \eta}{6} > 0$. So we may further assume that $f_{3}$ is a $5$-face. This means that $v$ is a $(3, 4, 5, 4)$-vertex. By \cref{Lem1}\ref{3454}, $v_{3}$ is a $(4^{+}, 4^{+}, 4^{+}, 4^{+})$-vertex. Note that a $(4^{+}, 4^{+}, 4^{+}, 4^{+})$-vertex is not adjacent to a bad $4$-vertex. By \ref{R7}, $v_{3}$ sends $\frac{1}{20}$ to $v$. Then $\omega^{*}(v) \geq -2 + \frac{1}{2} + \frac{2 - \eta}{3} + \frac{4}{5} + \frac{1}{20} = \frac{1 - 20\eta}{60}  > 0$. 
     
    \item Assume that $f_{3}$ is a $5^{+}$-face, and one of $f_{2}$ and $f_{4}$ is a $5^{+}$-face. Recall that $v$ is incident with a $4$-face, we may assume that $f_{2}$ is a $4$-face. If $f_{2}$ sends at least $\frac{1- \eta}{2}$ to $v$, then $\omega^{*}(v) \geq -2 + \frac{1 - \eta}{2} + \frac{4}{5} \times 2 = \frac{1 - 5\eta}{10} > 0$. So we may assume that $f_{2}$ is incident with four $4$-vertices, and one of which is a special $4$-vertex. Since there is no subgraph isomorphic to the configuration in \cref{House}, we have that $d(v_{1}) \geq 5$. By the discharging rules, $f_{2}$ sends at least $\frac{1}{3}$ to $v$, and $f_{3}$ sends at least $\frac{4}{5}$ to $v$, and $f_{4}$ sends at least $\frac{7}{8}$ to $v$. Then $\omega^{*}(v) \geq -2 + \frac{1}{3} + \frac{4}{5} + \frac{7}{8} > 0$. 
    
    \item Assume that $f_{3}$ is a $4$-face. Since $v$ is incident with at most one $6^{+}$-face, \cref{Lem1}\ref{3k4l} implies that $v$ is a $(3, 4, 4, 6^{+})$-vertex, and none of the two incident $4$-faces is incident with a special $4$-vertex. By the discharging rules, each incident $4$-face sends at least $\frac{1}{2}$ to $v$, and the $6^{+}$-face sends at least $1$ to $v$. Then \begin{equation}\label{EQ4}
    \omega^{*}(v) \geq -2 + \frac{1}{2} \times 2 + 1 = 0. 
\end{equation}
    \end{itemize}
\end{itemize}

We have proved that $\omega'(x) \geq 0$ for all $x \in V(G) \cup F(G)$, and every $5^{+}$-vertex has a positive final charge. Next, we show that there is a vertex $v$ such that $\omega'(v) > 0$. Suppose that $\omega'(v) = 0$ for all $v \in V(G)$. Then $G$ is $4$-regular, and no vertex is a bad $4$-vertex. This implies that $\omega'(v) = \omega^{*}(v)$ for all $v \in V(G)$, and no vertex is a $(4^{+}, 4^{+}, 4^{+}, 4^{+})$-vertex. By the above discussion, a $4$-vertex may have final charge zero only in the four labeled equations. 

In \eqref{EQ1}, $\omega'(v) = 0$ implies that $v$ is a $(3, 3, 4, 6^{+})$-vertex, then there is a subgraph isomorphic to the configuration in \cref{House}, this contradicts \cref{3CVC4}. 

In \eqref{EQ2}, $\omega'(v) = 0$ implies that $v$ is a $(3, 6, 3, 6)$-vertex. Since $v$ is a $4$-vertex, it is incident with each incident $6$-face once, then $v_{1}v_{2}v_{3}v_{4}v_{1}$ is a $4$-cycle. Since $G[\{v, v_{1}, v_{2}, v_{3}, v_{4}\}]$ contains no $K_{5}^{-}$, we have that $G[\{v, v_{1}, v_{2}, v_{3}, v_{4}\}]$ is an induced wheel, and there is a subgraph isomorphic to the configuration in \cref{Kite} (note that $G$ is $4$-regular), a contradiction. 

In \eqref{EQ3}, $\omega'(v) = 0$ implies that $v$ is a $(3, 3, 6, 6)$-vertex. Similar to the above, $G[\{v, v_{1}, v_{2}, v_{3}, v_{4}\}]$ is an induced wheel, and there is a subgraph isomorphic to the configuration in \cref{Kite}, a contradiction. 

In \eqref{EQ4}, $\omega'(v) = 0$ implies that $v$ is a $(3, 4, 4, 6)$-vertex. Then there is a subgraph isomorphic to the configuration in \cref{House}, a contradiction. 

Therefore, there is a vertex having positive final charge, this completes the proof.

\section{Proof of \cref{WD}}
\label{sec4}
To prove the main result stated in \cref{WD}, we need the following lemma. A \emph{GDP-tree} is a connected graph in which every block is either a cycle or a complete graph. 
\begin{lemma}[{\cite[Lemma 5.5]{MR4606413}}]\label{Gallai-Weak}
Assume $G$ is a graph which is not weakly $(h-1)$-degenerate. Let $U \subseteq \{u \in V(G) : d(v) = h(v)\}$. If $G - U$ is weakly $(h-1)$-degenerate, then every component of $G[U]$ is a GDP-tree.
\end{lemma}

\begin{proof}[Proof of \cref{WD}]
Assume that $G$ is a minimum counterexample to \cref{WD}, \ie $G$ is not weakly $3$-degenerate but every proper subgraph is weakly $3$-degenerate. Since $G$ is a minimum counterexample, it must be connected, and its minimum degree is at least four. Suppose that $G$ contains an induced subgraph isomorphic to the configuration in \cref{K5--}. Note that the induced subgraph is a $2$-connected subgraph that is neither a complete graph $K_{5}$ nor a cycle. This contradicts \cref{Gallai-Weak}, thus there is no induced subgraph isomorphic to the configuration in \cref{K5--}. Similarly, we can prove that there is no induced subgraph isomorphic to the configuration in \cref{Kite} or \cref{House}. This contradicts \cref{SR}.
\end{proof}

\section{Proof of \cref{AT4}}
\label{sec5}
To prove \cref{AT4}, we need the following lemma. 
\begin{lemma}[Lu \etal \cite{MR4051856}]\label{L}
Let $D$ be a digraph with $V(D) = X_{1} \cup X_{2}$ and $X_{1} \cap X_{2} = \emptyset$. If all arcs between $X_{1}$ and $X_{2}$ are oriented from $X_{1}$ to $X_{2}$, then $D$ is an AT-orientation if and only if both $D[X_{1}]$ and $D[X_{2}]$ are AT-orientations.  
\end{lemma}

\begin{proof}[Proof of \cref{AT4}]
Assume that $G$ is a counterexample to \cref{AT4} with $|V(G)|$ as small as possible. By \cref{SR}, there is an induced subgraph $\Gamma$ isomorphic to a configuration as described in \cref{SR}. By the minimality, $AT(G - V(\Gamma)) \leq 4$. By the definition of Alon-Tarsi number, there exists an AT-orientation $D'$ of $G - V(\Gamma)$ with maximum out-degree at most three. We can extend $D'$ to an AT-orientation $D$ of $G$ with maximum out-degree at most three. Firstly, we orient all the edges between $V(\Gamma)$ and $\overline{V(\Gamma)}$ from $V(\Gamma)$ to $\overline{V(\Gamma)}$. 

Suppose that $\Gamma$ is a vertex of degree at most three in $G$. The obtained orientation is an orientation $D$ of $G$. Then $\mathrm{diff}(D) = \mathrm{diff}(D') \neq 0$, and the maximum out-degree is at most three. Therefore, $AT(G) \leq 4$, a contradiction.

Suppose that $\Gamma$ is isomorphic to the configuration in \cref{K5--}, \cref{Kite} or \cref{House}. We orient all the edges in $\Gamma$ as depicted in \cref{Oriented}. Note that the orientation of $\Gamma$ is an AT-orientation, and the maximum out-degree of the orientation $D$ of $G$ is at most three. By \cref{L}, $D$ is an AT-orientation, and $AT(G) \leq 4$, a contradiction. 

This completes the proof of \cref{AT4}.
\end{proof}

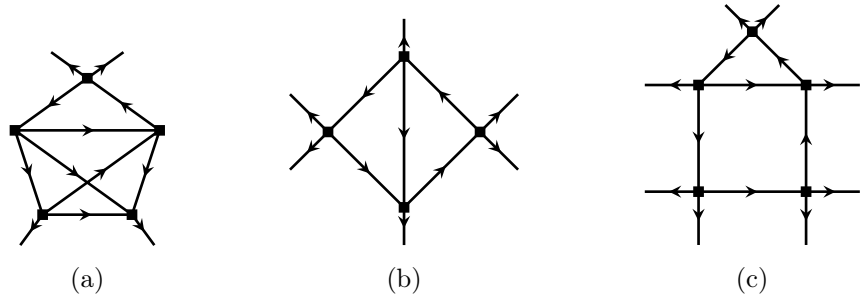
\begin{figure}%
\centering
\def\s{1}
\subcaptionbox{\label{ATA}}
{\begin{tikzpicture}[line width = 1pt]
\foreach \ang in {1, 2, 3, 4, 5}
{
\def\pointname{v\ang}
\coordinate (\pointname) at ($(\ang*360/5-54:\s)$);}

\path [draw=black, postaction={on each segment={mid arrow=black}}]
(v1)--(v2)--(v3)--(v4)--cycle
(v3)--(v1)
(v1)--(v5)
(v3)--(v5)
(v4)--(v5)
(v2)--($(v2)!0.5!180:(v1)$)
(v2)--($(v2)!0.5!-180:(v3)$)
(v4)--($1.5*(v4)$)
(v5)--($1.5*(v5)$)
;
\foreach \ang in {1, 2, 3, 4, 5}
{
\node[rectangle, inner sep = 1.5, fill = black, draw] () at (v\ang) {};
}
\end{tikzpicture}}\hspace{1.5cm}
\subcaptionbox{\label{ATB}}
{\begin{tikzpicture}[line width = 1pt]%
\coordinate (E) at (\s, 0);
\coordinate (N) at (0, \s);
\coordinate (W) at (-\s, 0);
\coordinate (S) at (0, -\s);

\path [draw=black, postaction={on each segment={mid arrow=black}}]
(N)--(W)--(S)--(E)--cycle
(N)--(S)
(W)--($(W)!0.5!180:(S)$)
(W)--($(W)!0.5!180:(N)$)
(E)--($(E)!0.5!180:(S)$)
(E)--($(E)!0.5!180:(N)$)
(N)--($1.5*(N)$)
(S)--($1.5*(S)$)
;
\node[rectangle, inner sep = 1.5, fill, draw] () at (E) {};
\node[rectangle, inner sep = 1.5, fill, draw] () at (N) {};
\node[rectangle, inner sep = 1.5, fill, draw] () at (W) {};
\node[rectangle, inner sep = 1.5, fill, draw] () at (S) {};
\end{tikzpicture}}\hspace{1.5cm}
\subcaptionbox{\label{ATC}}
{\begin{tikzpicture}[line width = 1pt]
\coordinate (A) at (45:\s);
\coordinate (B) at (135:\s);
\coordinate (C) at (225:\s);
\coordinate (D) at (-45:\s);
\coordinate (H) at (90:1.414*\s);

\path [draw=black, postaction={on each segment={mid arrow=black}}]
(A)--(H)--(B)--(C)--(D)--cycle
(B)--(A)
(A)--($(A)!0.5!180:(B)$)
(B)--($(B)!0.5!180:(A)$)
(C)--($(C)!0.5!180:(B)$)
(C)--($(C)!0.5!180:(D)$)
(D)--($(D)!0.5!180:(A)$)
(D)--($(D)!0.5!180:(C)$)
(H)--($(H)!0.5!180:(A)$)
(H)--($(H)!0.5!180:(B)$)
;

\node[rectangle, inner sep = 1.5, fill, draw] () at (A) {};
\node[rectangle, inner sep = 1.5, fill, draw] () at (B) {};
\node[rectangle, inner sep = 1.5, fill, draw] () at (C) {};
\node[rectangle, inner sep = 1.5, fill, draw] () at (D) {};
\node[rectangle, inner sep = 1.5, fill, draw] () at (H) {};
\end{tikzpicture}}
\caption{Orientations.}
\label{Oriented}
\end{figure}

\section{Concluding remark}
\label{sec6}
In \cite{MR3634481}, Choi provided two infinite families of toroidal graphs. One of them has no $K_{5}^{-}$ but has chromatic number $5$, while the other has no $6$-cycle but has chromatic number $5$. It is worth to note that $\textsf{ch}(G) \leq \textsf{wd}(G) + 1$ and $\textsf{ch}(G) \leq AT(G)$. The two families of toroidal graphs show that they are not weakly $3$-degenerate, and have Alon-Tarsi number greater than four. Then the conditions in \cref{WD,AT4} are sharp in some sense.

\cref{SR} can be applied to other graph parameters as well. A strictly $f$-degenerate transversal is a common generalization of DP-coloring and $L$-forested coloring. For the definitions of DP-coloring, $L$-forested coloring, and strictly $f$-degenerate transversal, we refer the reader to \cite{MR4357325}. Since the four configurations described in \cref{SR} are reducible configurations for the following result, we can easily obtain it as \cref{WD,AT4}.

\begin{theorem}
Let $G$ be a toroidal graph without $K_{5}^{-}$ and $6$-cycles. Let $H$ be a cover of $G$ and $f$ be a function from $V(H)$ to $\{0, 1, 2\}$. If $f(v, 1) + f(v, 2) + \dots + f(v, s) \geq 4$ for each $v \in V(G)$, then $H$ has a strictly $f$-degenerate transversal.  
\end{theorem}

As a consequence, one can obtain the following result on list vertex arboricity.

\begin{corollary}[Zhu \etal \cite{MR4364836}]
Let $G$ be a toroidal graph without $K_{5}^{-}$ and $6$-cycles. Then the list vertex arboricity is at most two. 
\end{corollary}

\vskip 0mm \vspace{0.3cm} \noindent{\bf Acknowledgments.} We thank the two anonymous referees for their valuable comments and constructive suggestions on the manuscript.  This research was supported by the Natural Science Foundation of Henan Province (No. 242300420238).


\begin{thebibliography}{10}

\bibitem{MR4606413}
A.~Bernshteyn and E.~Lee, Weak degeneracy of graphs, J. Graph Theory 103~(4)
  (2023) 607--634.

\bibitem{MR1722227}
T.~B\"{o}hme, B.~Mohar and M.~Stiebitz, Dirac's map-color theorem for
  choosability, J. Graph Theory 32~(4) (1999) 327--339.

\bibitem{MR2682511}
L.~Cai, W.~Wang and X.~Zhu, Choosability of toroidal graphs without short
  cycles, J. Graph Theory 65~(1) (2010) 1--15.

\bibitem{MR3634481}
I.~Choi, Toroidal graphs containing neither {$K^-_5$} nor 6-cycles are
  4-choosable, J. Graph Theory 85~(1) (2017) 172--186.

\bibitem{MR3758240}
Z.~Dvo\v{r}\'{a}k and L.~Postle, Correspondence coloring and its application to
  list-coloring planar graphs without cycles of lengths 4 to 8, J. Combin.
  Theory Ser. B 129 (2018) 38--54.

\bibitem{MR4152773}
J.~Grytczuk and X.~Zhu, The {A}lon-{T}arsi number of a planar graph minus a
  matching, J. Combin. Theory Ser. B 145 (2020) 511--520.

\bibitem{MR4663366}
M.~Han, T.~Wang, J.~Wu, H.~Zhou and X.~Zhu, Weak degeneracy of planar graphs
  and locally planar graphs, Electron. J. Combin. 30~(4) (2023) Paper No. 4.18.

\bibitem{MR4357325}
F.~Lu, Q.~Wang and T.~Wang, Cover and variable degeneracy, Discrete Math.
  345~(4) (2022) 112765.

\bibitem{MR4051856}
H.~Lu and X.~Zhu, The {A}lon-{T}arsi number of planar graphs without cycles of
  lengths 4 and {$l$}, Discrete Math. 343~(5) (2020) 111797.

\bibitem{MR2515754}
U.~Schauz, Mr. {P}aint and {M}rs. {C}orrect, Electron. J. Combin. 16~(1) (2009)
  P\#77.

\bibitem{MR2578908}
U.~Schauz, Flexible color lists in {A}lon and {T}arsi's theorem, and time
  scheduling with unreliable participants, Electron. J. Combin. 17~(1) (2010)
  P\#13.

\bibitem{MR1290638}
C.~Thomassen, Every planar graph is $5$-choosable, J. Combin. Theory Ser. B
  62~(1) (1994) 180--181.

\bibitem{MR1235909}
M.~Voigt, List colourings of planar graphs, Discrete Math. 120~(1-3) (1993)
  215--219.

\bibitem{Wang2019+}
Q.~Wang, T.~Wang and X.~Yang, Variable degeneracy of graphs with restricted
  structures, arXiv:2112.09334,
  \url{https://doi.org/10.48550/arXiv.2112.09334}.

\bibitem{MR4564473}
T.~Wang, Weak degeneracy of planar graphs without 4- and 6-cycles, Discrete
  Appl. Math. 334 (2023) 110--118.

\bibitem{MR4364836}
A.~Zhu, D.~Chen, M.~Chen and W.~Wang, Vertex-arboricity of toroidal graphs
  without {$K_5^-$} and 6-cycles, Discrete Appl. Math. 310 (2022) 97--108.

\bibitem{MR3906645}
X.~Zhu, The {A}lon-{T}arsi number of planar graphs, J. Combin. Theory Ser. B
  134 (2019) 354--358.

\end{thebibliography}
\end{document}